\documentclass[oneside,12pt]{amsart}

\usepackage{amscd,amsmath,latexsym,bm,pdfpages,verbatim}
\usepackage[mathcal]{euscript}
\usepackage{graphicx}
\usepackage{amssymb}          
\usepackage{epstopdf}
\usepackage[all]{xy}

\newtheorem{thm}{Theorem}[section]

\newtheorem{lem}[thm]{Lemma}
\newtheorem{prop}[thm]{Proposition}

\theoremstyle{definition}

\theoremstyle{definition}
\newtheorem{construction}[thm]{Construction}

\theoremstyle{definition}

\theoremstyle{remark}
\newtheorem*{rem}{Remark}

\newcommand{\R}{{\mathbb R}}
\newcommand{\C}{{\mathbb C}}

\begin{document}
\def\X#1#2{r(v^{#2}\ds{\prod_{i \in #1}}{x_{i}})}
\def\skp#1{\vskip#1cm\relax}
\def\C{{\mathbb C}}
\def\R{{\mathbb R}}
\def\ce{{\mathbb C}}
\def\erre{{\mathbb R}}
\def\efe{{\mathbb F}}
\def\ene{{\mathbb N}}
\def\UNO{1\mkern-7mu1}
\def\ee{{\mathbb E}}
\def\pee{{\mathbb P}}
\def\ene{{\mathbb N}}
\def\tn{{\bf T$^{n}$}}
\def\pn{$P^{n}\/$}
\def\cn{{\bf C$^{n}$}}
\def\z2{{\bf Z}$_{2}$}
\def\zl2{{\bf Z}$_{(2)}$}
\def\block{\rule{2.4mm}{2.4mm}}
\def\rplus{{\bf R$_{+}}}
\def\nd{\noindent}
\def\becomes{\colon\hspace{-2,5mm}=}
\def\ds{\displaystyle}
\def\s{\sigma}
\def\ptheta{\mathbf{Z}[\theta_1,\cdots, \theta_n]}

\numberwithin{equation}{section}
\title{On monomial ideal rings and a theorem of Trevisan}
\skp{0.2}
\author[A.~Bahri]{A.~Bahri}
\address{Department of Mathematics,
Rider University, Lawrenceville, NJ 08648, U.S.A.}
\email{bahri@rider.edu}

\author[M.~Bendersky]{M.~Bendersky}
\address{Department of Mathematics
Department of Mathematics, Hunter College, East 695 Park Avenue, New
York, NY 10065, U.S.A.}
\email{mbenders@hunter.cuny.edu}

\author[F.~R.~Cohen]{F.~R.~Cohen}
\address{Department of Mathematics,
University of Rochester, Rochester, NY 14625, U.S.A.}
 \email{cohf@math.rochester.edu}


\author[S.~Gitler]{S.~Gitler}
\address{El Colegio Nacional, Gonzalez Obregon 24\,C, Centro Historico, Mexico City, Mexico.}  \email{sgitler@math.cinvestav.mx}



\thanks{A.\,B.\ was supported in part by a Rider University Summer Research Fellowship and 
grant number 210386 from the Simons Foundation; F.\,R.\,C.\ was supported partially by 
DARPA grant number 2006-06918-01.}



\begin{abstract}
A direct proof is presented of a form of Alvise Trevisan's result \cite{alvise}, that every monomial ideal ring is represented by the cohomology of a topological space. Certain of these rings are shown to be
realized by  polyhedral products indexed by simplicial complexes. 
\end{abstract}

\

\maketitle

\maketitle
\section{Introduction}
In the paper \cite{alvise}, Alvise Trevisan showed that every ring which is a quotient of an integral polynomial ring
by an ideal of monomial relations, can be realized as the integral cohomology ring of a topological space.
Moreover, he showed that the rings could be all realized with spaces which are generalized Davis-Januszkiewicz
spaces. These spaces are colimits over {\em multicomplexes\/} which are generalizations of simplicial complexes.

Here is presented a  direct proof of the ``realization'' part of Trevisan's theorem. It uses a result of
Fr\"{o}berg from \cite{fr} which asserts that a map known as ``polarization'' produces in a natural way, a 
regular sequence of degree-two elements. This allows for the realization of any monomial ideal ring by a 
certain pullback.

It is noted also that certain families of monomial ideal rings, beyond Stanley-Reisner rings, can be 
realized as generalized Davis-Januszkiewicz spaces based on ordinary simplicial complexes. Of course,
as Trevisan shows, multicomplexes are needed in general.

Through the paper, all cohomology is taken with {\em integral\/} coefficients.

\section{The main result}

Let $\mathbb{Z}[x_{1},\ldots,x_{n}]$ be a polynomial ring and 
\begin{equation}\label{eqn:m}
M = \big\{m_{j}\big\}_{j=1}^{r}, \qquad m_{j} = x_{1}^{t_{1j}}x_{2}^{t_{2j}}\cdots x_{n}^{t_{nj}}
\end{equation}

\nd be a set of minimal monomials, that is, no monomial divides another. Here, the exponent 
$t_{ij}$ might 
be equal to zero but every $x_i$ must appear in some $m_j$. Notice that 
the set $M$ is determined by the $n\times r$ matrix $(t_{ij} )$.
Denote by $I(M)$ the ideal in $\mathbb{Z}[x_{1},\ldots,x_{n}]$ generated
by the minimal monomials $m_{j}$ and set
\begin{equation}\label{eqn:mir}
A = A(M) =\mathbb{Z}[x_{1},\ldots,x_{n}]\big/I(M)
\end{equation}  

\nd a {\it monomial ideal ring\/}. From this is defined a second monomial ideal ring $A(\overline{M})$
with monomial ideal generated by square free monomials. For each $i = 1,2,\ldots,n$ 
set 
\begin{equation}\label{eqn:ti}
t_{i} = \text{max}\{t_{i1},t_{i2},\ldots,t_{ir}\}
\end{equation}

\nd the largest entry in the $i$-th row of $(t_{ij})$.   Next, introduce new 
variables $y_{i1},y_{i2},\ldots,y_{it_{i}}$ for each $i = 1,2,\ldots,n$.  For each monomial
$m_{j} = x_{1}^{t_{1j}}x_{2}^{t_{2j}}\cdots x_{n}^{t_{nj}}$, set
\begin{equation}\label{eqn:order}
\overline{m}_{j} = (y_{11}y_{12}\cdots y_{1t_{1j}})(y_{21}y_{22}\cdots y_{2t_{2j}})\;
\cdots \;(y_{n1}y_{n2}\cdots y_{nt_{nj}}).
\end{equation}

\nd Let $\overline{M} = \{\overline{m}_{j}\}_{j=1}^{r}$ and define an algebra $B$ by
\begin{equation}\label{eqn:am}
B = B(\overline{M}) = \mathbb{Z}[y_{11},y_{12},\ldots,y_{1t_{1}},y_{21},y_{22},\ldots,y_{2t_{2}}, \; \ldots\;,
y_{n1},y_{n2},\ldots,y_{nt_{n}}]\big/I(\overline{M}).
\end{equation}

\nd The monomials here are square-free so $B$ is a Stanley-Reisner algebra which
determines a simplicial complex $K(\overline{M})$. (This process which constructs $B$ from $A$
is known in the literature as {\em polarization}.) Associated to this simplicial complex is a fibration
$$Z\big(K(\overline{M});(D^{2},S^{1})\big) \longrightarrow \mathcal{DJ}(K(\overline{M})) \longrightarrow
BT^{d(\overline{M})}$$

\nd where $d(\overline{M}) = \sum_{1=1}^{n}{t_{i}}\;$, with $t_{i}$ as in \eqref{eqn:ti}, $\mathcal{DJ}(K(\overline{M}))$ is the 
Davis-Januszkiewicz space of the simplicial
complex $K(\overline{M})$ and $Z\big(K(\overline{M});(D^{2},S^{1})\big)$ is the moment-angle complex corresponding to  $K(\overline{M})$, (\cite{bp}).  Recall that the Davis-Januszkiewicz space has the property that
\begin{equation}\label{eqn:sr}
H^{*}\big(\mathcal{DJ}(K(\overline{M}))\big) \cong B.
\end{equation}

\nd Define next a diagonal map $\Delta\colon T^{n} \longrightarrow T^{d(\overline{M})}$ by
\begin{equation}\label{eqn:delta}
\Delta(x_{1},x_{2},\ldots,x_l) = \big(\Delta_{t_{1}}(x_{1}), \Delta_{t_{2}}(x_{2}),\ldots,\Delta_{t_{n}}(x_l)\big)                                 
\end{equation}

\nd where $\Delta_{t_{i}}(x_{i}) = (x_{i},x_{i},\ldots,x_{i}) \in T^{t_{i}}$. In the diagram below,
let $W(A)$ be defined as the pullback of the fibration.

\begin{equation}\label{eqn:manifold.diagram}
\begin{CD}
Z\big(K(\overline{M});(D^{2},S^{1})\big) @>>{=}>
Z\big(K(\overline{M});(D^{2},S^{1})\big)\\
@VV{}V              @V{}VV    \\
W(A) @>>{\widetilde{\Delta}}>
\mathcal{DJ}(K(\overline{M}))\\
@VV{}V              @VV{}V    \\
BT^{n} @>{B\Delta}>>BT^{d(\overline{M})}
\end{CD}
\end{equation}

\

\nd The diagram \eqref{eqn:manifold.diagram} 
extends to to the right and produces the fibration  

\begin{equation}\label{eqn:fib}
W(A) \stackrel{{\widetilde{\Delta}}}{\longrightarrow} 
\mathcal{DJ}(K(\overline{M}))  \stackrel{p}{\longrightarrow} BT^{d(\overline{M})-n}.
\end{equation}

\

\nd Recall that  $d(\overline{M}) = \sum_{1=1}^{n}{t_{i}}$ and choose generators

$$H^{\ast}(BT^{d(\overline{M})-n})\; \cong \;
\mathbb{Z}[u_{12},\ldots,u_{1t_{1}},u_{22},\ldots,u_{2t_{2}}, \; \ldots\;,
u_{n2},\ldots,u_{nt_{n}}],$$

\nd so that 
$$p^{*}(u_{ik_{i}})\; = \;y_{i1}-y_{ik_{i}}\quad i = 1,2,\ldots,n, \;\; k_{i} = 2,3,\ldots t_{i}.$$

\

\nd The fact that $p^{*}$ is determined by the diagonal $\Delta$ in diagram  \eqref{eqn:manifold.diagram},
allows this choice. Set $\theta_{ik_{i}}  \becomes p^{*}(u_{ik_{i}})$. The proposition following is a basic result
about the diagonal map $\Delta$, (the {\em polarization\/} map);  a proof may be found in 
\cite[page 30]{fr}.

\begin{prop}[Fr\"{o}berg]\label{prop:reg}
Over any field $k$, the sequence $\{\theta_{ik_{i}}\}$ is a 
regular sequence of degree-two elements in the ring $H^{*}\big(\mathcal{DJ}(K(\overline{M}));k\big)$. 
 \end{prop} 

\

This result allows for a direct proof of the realization theorem.
\begin{thm}\label{thm:dj}
There is an isomorphism of rings
$$H^{*}\big(W(A);\mathbb{Z}\big) \longrightarrow A(M).$$
\end{thm}

\begin{proof}
Working over a field $k$ and following Masuda-Panov, \cite[Lemma 2.1]{mp}, we use the Eilenberg-Moore 
spectral sequence associated to the 
fibration \eqref{eqn:fib}. It has

$$E_{2}^{\ast,\ast} = \rm{Tor}^{\ast,\ast}_{H^{\ast}(BT^{d(\overline{M})-n})}
(H^{*}\big(\mathcal{DJ}(K(\overline{M})),k\big).$$

\

\nd Now $H^{*}\big(\mathcal{DJ}(K(\overline{M}))\big)$ is free as an 
$H^{\ast}(BT^{d(\overline{M})-n})$-module by Proposition \ref{prop:reg}, so
\begin{align*}
\rm{Tor}^{\ast,\ast}_{H^{\ast}(BT^{d(\overline{M})-n})}
(H^{*}\big(\mathcal{DJ}(K(\overline{M}))\big),k) &= 
\rm{Tor}^{0,\ast}_{H^{\ast}(BT^{d(\overline{M})-n})}
(H^{*}\big(\mathcal{DJ}(K(\overline{M}))\big),k)\\
&=  H^{*}\big(\mathcal{DJ}(K(\overline{M}))\big) \otimes_{H^{\ast}(BT^{d(\overline{M})-n})}k\\
&=  H^{*}\big(\mathcal{DJ}(K(\overline{M}))\big)\big/p^{*}(H^{>0}(BT^{d(\overline{M})-n})).
\end{align*}

\nd It follows that the Eilenberg-Moore spectral sequence collapses at the $E_{2}$ term and hence, as groups,
$$H^{*}\big(W(A)\big) =  H^{*}\big(\mathcal{DJ}(K(\overline{M}))\big)\big/p^{*}(H^{>0}(BT^{d(\overline{M})-n}))$$

\nd from which we conclude that $H^{*}\big(W(A);k\big)$ is concentrated in even degree. 
Taking $k = \mathbb{Q}$  gives the result that in odd degree, $H^{*}\big(W(A);\mathbb{Z}\big)$ consists
of torsion only. Unless this torsion is zero, the argument above with $k = \mathbb{F}_{p}$ for an appropriate $p$,
implies a contradiction. It follows that $H^{*}\big(W(A);\mathbb{Z}\big)$ is concentrated in even degree.

\begin{lem}\label{lem:serress}
The integral Serre spectral sequence of the fibration \eqref{eqn:fib} collapses.
\end{lem}

\begin{proof}
The spaces in the fibration have integral cohomology concentrated in even degrees.
\end{proof}

\

\nd The $E_{2}$-term of the Serre spectral sequence is
$$H^{*}\big(W(A); \mathbb{Z}\big) \otimes 
H^{*}\big(BT^{d(\overline{M})-n}; \mathbb{Z}\big).$$

\nd It follows that, as a ring, $H^{*}\big(W(A); \mathbb{Z}\big)$ is the quotient of 
$H^{*}\big(\mathcal{DJ}(K(\overline{M}))\big)$ by the two-sided ideal $L$ generated by
the image of $p^{*}$. So there is an isomorphism of graded rings,

$$H^{*}\big(W(A); \mathbb{Z}\big) \longrightarrow 
H^{*}\big(\mathcal{DJ}(K(\overline{M}))\big)\big/L \cong A(\overline{M})\big/L \cong A(M)$$

\

\nd completing the proof of Theorem \ref{thm:dj}. \end{proof}

\section{On the geometric realization of certain monomial ideal rings by ordinary polyhedral products}
In this section,  polyhedral products, \cite{bbcg},  involving  finite
and infinite complex projective spaces are used to realize certain classes of monomial ideal rings.
As noted earlier, generalizations of the Davis-Januszkiewicz spaces to the realm of multicomplexes 
are required in order to realize all monomial ideal rings, see Trevisan \cite{alvise}.

The class which can be realized by ordinary polyhedral products is restricted to those
monomials 
$$M = \big\{m_{j}\big\}_{j=1}^{r}, \qquad m_{j} = x_{1}^{t_{1j}}x_{2}^{t_{2j}}\cdots x_{n}^{t_{nj}}$$

\nd of \eqref{eqn:m}, which satisfy the condition:

\begin{itemize}
\item[$\divideontimes$] $t_{ij}$ is constant over all monomials  $m_{j}$ which have
$t_{ij}$ and {\em at least one other exponent\/} both non-zero.
\end{itemize}

\nd In particular, a monomial ring of the form 
\begin{equation}\label{eqn:ex}
\mathbb{Z}[x_{1},x_{2},x_{3}]\big/\langle x_{1}^{2}x_{2},x_{1}^{2}x_{3}^{4}, x_{3}^{5}\rangle
\end{equation}

\nd {\em can\/} be realized by an ordinary polyhedral product.
As usual, let $(\underline{X},\underline{A})$ denote a family of CW pairs
$$\{(X_{1},A_{1}), (X_{2},A_{2}), \ldots,(X_{n},A_{n})\}.$$

\nd Given a monomial ring $A(M)$ of the form \eqref{eqn:mir}, satisfying the condition $\divideontimes$ above,
a simplicial complex $K$ and a family of pairs $(\underline{X},\underline{A})$ will be specified so that
$$H^{*}\big(Z(K; (\underline{X},\underline{A})); \mathbb{Z}\big) = A(M)$$

\nd where $Z(K; (\underline{X},\underline{A})$ represents a polyhedral product as defined in \cite{bbcg}.
\begin{construction}\label{con:k}
Let $K$ be the simplicial complex on  $n$ vertices $\{v_{1},v_{2},\ldots,v_{n}\}$  which has
a minimal non-face corresponding to
each $m_{i}$ having {\em at least two\/} non-zero exponents. If $m_{i}$ has non-zero exponents
$$t_{j_{1}i},t_{j_{2}i},\ldots, t_{j_{t}i}$$

\nd then $K$ will have a corresponding minimal non-face $\{v_{j_{1}},v_{j_{2}},\ldots,v_{j_{t}}\}$. 
Moreover, these will be the only minimal non-faces of $K$. 
\end{construction}

\nd For example, the ring \eqref{eqn:ex} above will have associated to it, the
simplicial complex $K$ on vertices $\{v_{1},v_{2}, v_{3}\}$ and will have minimal non-faces 
$\{v_{1},v_{2}\}$ and $\{v_{1},v_{3}\}$. So, $K$ will be the disjoint union of a point and a one-simplex.

\

For the set of monomials $M$ satisfying condition $\divideontimes$, the cases 
following are distinguished in terms of \eqref{eqn:m} for fixed $i \in \{1,2,\ldots,n\}$.
\begin{enumerate}
\item For certain $j$, $t_{ij} = 1$, $t_{i'j}\neq 0$ for some $i' \neq i$ and $t_{ik}=0$ otherwise. 
\item For certain $j$, $t_{ij} = q_{i} > 1$, $t_{i'j}\neq 0$ for some $i' \neq i$ and $t_{ik}=0$ otherwise. 
\item $m_{j} = x_{i}^{s_{i}}$ for some $j$ and $t_{ik} = 0$ for $k\neq j$.
\item $m_{j} = x_{i}^{s_{i}}$ for some $j$ and  if $t_{ik} \neq 0$ for $k\neq j$, then $t_{ik} = q_{i} < s_{i}$.
\end{enumerate}

\

\nd With this classification in mind, define a family of CW-pairs 

$$(\underline{X},\underline{A}) = \{(X_{i},A_{i})\colon i = 1,\ldots,n\}$$

\nd by
\begin{equation}\label{eqn:cases}(X_{i},A_{i}) \quad = \quad \begin{cases}
\;(\mathbb{C}P^{\infty},\ast)\qquad  \;\;\;\text{if}\; i \;\text{satisfies}\; (1),\\
\;(\mathbb{C}P^{\infty},\mathbb{C}P^{q_{i}-1})\qquad \;\;\;\text{if}\; i \;\text{satisfies}\; (2),\\
\;(\mathbb{C}P^{s_{i}-1},\ast) \qquad \text{if}\; i \;\text{satisfies}\; (3),\\
\;(\mathbb{C}P^{s_{i}-1},\mathbb{C}P^{q_{i}-1}) \qquad  \text{if}\; i \;\text{satisfies}\; (4).\\	
\end{cases}
\end{equation}

\

\nd The next theorem describes the polyhedral products which have cohomology realizing the monomial ideal rings
satisfying condition $\divideontimes$.
\begin{thm}\label{thm:pp}
Let $A(M)$ be a monomial ring of the form \eqref{eqn:mir}, satisfying the condition $\divideontimes$ and
$K$, the simplicial complex defined by Construction \ref{con:k}, then 
$$H^{*}\big(Z(K; (\underline{X},\underline{A})); \mathbb{Z}\big) = A(M)$$

\nd where $(X,A)$ is the pair specified by \eqref{eqn:cases}.
\end{thm}

\nd \begin{rem} The improvement here over \cite[Theorem 10.5]{bbcg3} consists of the inclusion of
cases (3) and (4) above. The polyhedral products which realize the monomial ideal rings discussed
in \cite{bbcg} have $X_{i} = \mathbb{C}P^{\infty}$ for all $i = 1,2,\ldots,n$.
\end{rem}

\begin{proof}[Proof of Theorem \ref{thm:pp}]
Set $Q = (q_{1},q_{2},\ldots, q_{n})$ with $q_{i}\geq 1$ for all $i$ and write the spaces $A_{i}$
of \eqref{eqn:cases} as $\mathbb{C}P^{q_{i}-1}$ where $q_{i}=1$ if $A_{i}= \ast$, a point. Write

$$(\underline{X},\underline{A}) = (\underline{X},\underline{\mathbb{C}P}^{Q-1}) =
\{(X_{i},\mathbb{C}P^{q_{i}-1})\colon i=1.2.\ldots,n\}$$

\

\nd and consider the commutative diagram

\begin{equation}\label{eqn:basic.diagram}
\begin{CD}
H^{*}(\textstyle{\prod_{i=1}^{n}}X_{i}) @<<{p^{*}}<
H^{*}(\textstyle{\prod_{i=1}^{n}}\mathbb{C}P^{\infty})\\
@VV{i^{*}}V              @V{k^{*}}VV    \\
H^{*}\big(Z(K;(\underline{X},\underline{\mathbb{C}P}^{Q-1}))\big) @<{h^{*}}<<
H^{*}\big(Z(K;(\underline{\mathbb{C}P}^{\infty},\underline{\mathbb{C}P}^{Q-1}))\big)
\end{CD}
\end{equation}

\

\nd induced by the various inclusion maps. According to \cite[Theorem 10.5]{bbcg3}, there is an
isomorphism of rings
$$H^{*}\big(Z(K;(\underline{\mathbb{C}P}^{\infty},\underline{\mathbb{C}P}^{Q-1}))\big)\;
\longrightarrow\; \mathbb{Z}[x_{1},\ldots,x_{n}]\big/I(M^{Q})$$

\nd where $I(M^{Q})$ is the ideal generated by all monomials 
$x_{i_{1}}^{q_{i_{1}}},x_{i_{2}}^{q_{i_{2}}},\ldots,x_{i_{k}}^{q_{i_{k}}}$ corresponding to the minimal
non-faces $\{v_{i_{1}},v_{i_{2}},\ldots ,v_{i_{k}}\}$ of $K$. Moreover, the proof of 
\cite[Lemma 10.3]{bbcg3} shows that the composition $i^{*}p^{*}$ is a surjection.
The commutativity of diagram 
\eqref{eqn:basic.diagram} implies that these relations all hold in 
$H^{*}\big(Z(K;(\underline{X},\underline{\mathbb{C}P}^{Q-1}))\big)$. In addition to these, the relation
$x_{i}^{s_{i}} = 0$ is included for each $i$ satisfying $X_{i} = \mathbb{C}P^{s_{i}-1}$. These relations
account for all the relations determined by $I(M)$. The remainder of the argument shows that $I(M)$
determines all relations in $H^{*}\big(Z(K; (\underline{X},\underline{A})); \mathbb{Z}\big)$.
Consider now the space

$$W_k  \;=\;  \mathbb{C}P^{q_1-1} \times \cdots \times \mathbb{C}P^{q_{k-1}-1} \times
X_{k} \times \mathbb{C}P^{q_{k+1}-1} \times \cdots 
\times \mathbb{C}P^{q_{n}-1}$$

\

\nd corresponding to the simplex $\{v_{k}\} \in K$, consisting of a single vertex.  The composition 

$$W_k \longrightarrow Z(K;(\underline{X}, \underline{\mathbb{C}P}^{Q-1})) 
\longrightarrow \textstyle{\prod_{i=1}^{n}X_{i}}$$ 

\

\nd factors the natural inclusion 
$W_k \longrightarrow \prod_{i=1}^{n}X_{i}$. From this observation follows the fact that
no other monomial relations occur in $H^{*}\big(Z(K;(\underline{X},\underline{\mathbb{C}P}^{Q-1}))\big)$
other than those determined by $I(M)$. Suppose next that there is a linear relationship of the form

\begin{equation}\label{eqn:relation}
a\omega = \sum_{i=1}^{k}a_{i}\omega_{i}
\end{equation}

\

\nd where $a, a_{i }\in \mathbb{Z}$ and $\omega, \omega_{i}$ are monomials in the 
$x_{i}, i = 1,2,\ldots,n$. Without loss of generality, $\omega$ and  $\omega_{i}$ can be assumed to be
not divisible by any of the monomials in $M$. Suppose 
$\omega = x_{j_{1}}^{\lambda_{1}}x_{j_{2}}^{\lambda_{2}}\cdots x_{j_{l}}^{\lambda_{l}}$, then
$\sigma = \{v_{j_{1}}, v_{j_{2}},\ldots,v_{j_{l}}\} \in K$ is a simplex and so is a full subcomplex of
$K$. (The corresponding polyhedral product $Z(\sigma;(\underline{X},\underline{\mathbb{C}P}^{Q-1}))$
is a product of finite and infinite complex projective spaces.)
This implies, by \cite[Lemma 2.2.3]{ds}, that 
$H^{*}\big(Z(\sigma;(\underline{X},\underline{\mathbb{C}P}^{Q-1}))\big)$ must be a direct
summand in $H^{*}\big(Z(K;(\underline{X},\underline{\mathbb{C}P}^{Q-1}))\big)$ contradicting the relation \eqref{eqn:relation}.\end{proof}

\bibliographystyle{amsalpha}

\end{document}